\documentclass[12pt, reqno]{amsart}
\usepackage{amsmath, amsthm, amscd, amsfonts, amssymb, graphicx, color}
\usepackage[bookmarksnumbered, colorlinks, plainpages]{hyperref}
\hypersetup{colorlinks=true,linkcolor=red,
anchorcolor=green, citecolor=cyan, urlcolor=red, filecolor=magenta, pdftoolbar=true}

\allowdisplaybreaks 

\newtheorem{theorem}{Theorem}[section]
\newtheorem{lemma}[theorem]{Lemma}
\newtheorem{proposition}[theorem]{Proposition}
\newtheorem{corollary}[theorem]{Corollary}
\theoremstyle{definition}

\newtheorem{notation}[theorem]{Notation}
\newtheorem{example}[theorem]{Example}

\newtheorem{remark}[theorem]{Remark}
\newtheorem{note}[theorem]{Note}
\numberwithin{equation}{section}

\begin{document}

\title[Approximately Dual frames under perturbations]{
The effect of perturbations of frames on their alternate and approximately dual frames}

\author[M. Hajiabootorabi, H. Javanshiri and M. R. Mardanbeigi]{M. Hajiabootorabi$^{1}$, H. Javanshiri$^{2}$ and M. R. Mardanbeigi$^{3}$}

\address{$^{2}$Department of Mathematics, Yazd University,
P.O. Box: 89195-741, Yazd, Iran }
\email{\textcolor[rgb]{0.00,0.00,0.84}{h.javanshiri@yazd.ac.ir;}}
\address{$^{1,3}$ Department of Mathematics, Science and Research Branch, Islamic Azad University, Tehran, Iran}
\email{\textcolor[rgb]{0.00,0.00,0.84}{mhajiabootorabi@yahoo.com;}}
\email{\textcolor[rgb]{0.00,0.00,0.84}{mrmardanbeigi@srbiau.ac.ir.}}




\begin{abstract}
Approximately dual frames as a generalization of duality notion in Hilbert spaces have applications in Gabor systems, wavelets, coorbit theory
and sensor modeling. In recent years, the computing of 
the associated deviations of the canonical and alternate dual frames from the original
ones has been considered by some authors. However, the quantitative measurement of the associated deviations of the alternate and approximately dual frames from the original
ones has not been satisfactorily answered. In this paper, among other things, it is proved that if the sequence $\Psi=(\psi_n)_n$
is sufficiently close to
the frame $\Phi=(\varphi_n)_n$, then $\Psi$ is a frame for $\mathcal H$
and approximately dual frames
$\Phi^{ad}=(\varphi^{ad}_n)_n$ and $\Psi^{ad}=(\psi^{ad}_n)_n$ can be found
which are close to each other and particularly, we
estimate the deviation from perfect reconstruction
in terms of the operator ${\mathcal A}_1:=T_\Phi
U_{\Phi^{ad}}$ and ${\mathcal A}_2:=T_\Psi U_{\Psi^{ad}}$ and their approximation rates, where $T_X$ and $U_X$
denote the synthesis
and analysis operators of the frame $X$, respectively. Finally, we
demonstrate how our results apply in the practical case of Gabor systems.
It is worth mentioning that some of our perturbation conditions are quite different
from those used in the previous literatures on this topic.
\end{abstract} 
\maketitle


\section{Introduction}

Frames were first introduced by Duffin and Schaeffer \cite{DSCH}. Today they are wide application throughout mathematics and engineering.
This is because of, a frame provides robust, unconditionally convergent, basis-like but usually non-unique expansions of elements of underlying Hilbert space.
Specifically, reconstruction of the original vector from frame is typically achieved by using a so-called duality
notion. A number of variations and generalizations of duality notion can be found in \cite{li1,li2} in the more general context of pseudo-duality, atomic system for subspace \cite{ATF}, approximate duality \cite{app,japp} and generalized duality \cite{app2,japp} .

In many situations,
it is important to know which properties of frames are
stable if we slightly modify the elements of the systems. This gives rise to the
so-called perturbation theory.
In detail, due to the fundamental works done by Paley and Wiener \cite{PWP}, the idea of a specific
perturbation of the typical exponential orthonormal basis of $L^2[-a,a]$
were formally introduced and popularized from then on. Various generalizations of Paley and Wiener perturbation Theorem have been appeared in the literature. For example, in \cite{BADA}, it was studied how perturbations of a frame sequence affect the
canonical dual. A similar approach was made in \cite{aref2} for G-frames.
Later on, the question of stability of duals with respect to perturbations in a much more general setting has been considered by Kutyniok and her coauthors in \cite{gita}. Among other things, they quantitatively measured this stability by considering
the associated deviations of the canonical and alternate dual sequences from the original
ones. Lately in \cite{japp}, the second author of the present paper, studied how perturbations of a frame affect the approximately dual frames. More precisely, he showed that if the sequence $\Psi=(\psi_n)_n$
is sufficiently close to
the frame $\Phi=(\varphi_n)_n$, then $\Psi$ is a frame for $\mathcal H$
and approximately dual frames
$\Phi^{ad}=(\varphi^{ad}_n)_n$ and $\Psi^{ad}=(\psi^{ad}_n)_n$ can be found
which are close to each other and $T_\Phi
U_{\Phi^{ad}}=T_\Psi U_{\Psi^{ad}}$, where $T_X$ and $U_X$
denote the synthesis
and analysis operators of the frame $X$, respectively.
It is worth mentioning that other observations on the approximately dual frames were investigated also in \cite{japp}.

We have two main goals in this paper. We first study some properties of (canonical) approximately dual frames which have not been touched so far. We then show that if we do a sufficiently small perturbation of a frame in the sense of \cite{per1,CHEN,per2}, the approximately dual of the new frame is also a small perturbation of
the approximately dual of the first one. In contrast to the work \cite{japp}, we consider the problem in a much more general setting. More precisely, we removed the imposed condition $T_\Phi
U_{\Phi^{ad}}=T_\Psi U_{\Psi^{ad}}$ and particularly, we
estimate for these cases the deviation from perfect reconstruction. These results pave the way for estimating the associated deviations of the canonical  [resp. alternate] approximately dual frames from the original ones which may be canonical  [resp. alternate] duals.


\section{Basic notations}

Throughout this paper, we denote by ${\mathcal H}$ a separable Hilbert
space with the inner product ``$\big<\cdot,\cdot\big>"$, the norm $\|\cdot\|$ and orthonormal basis $\xi=(e_n)_n$ and we use the set of natural numbers $\Bbb N$ as a generic index set for sequences and series in $\mathcal H$.
The notation $\ell^2$ is used
to denote the space of all square summable sequences on $\Bbb N$ equipped with the norm $\|\cdot\|_{\ell^2}$ and $\Delta=(\delta_n)_n$ refers to the canonical orthonormal basis of $\ell^2$. Furthermore, the notation $B({\mathcal H})$
[respectively, $B({\mathcal H},\ell^2)$] is used to denote the Banach space
of all bounded linear operators from $\mathcal H$ into $\mathcal
H$ [respectively, $\ell^2$].
For an operator $T\in B({\mathcal H})$,
the notations ${\mathcal R}(T)$ and $\ker(T)$ are used to denote the range and the
kernel of $T$, respectively; the notation $\|T\|_{\rm op}$ indicates the operator norm;
for closed subspace $\mathcal X$ of $\mathcal H$,
the letter $T|_{\mathcal X}$ refers to the restriction of $T$ to ${\mathcal X}$ and    $P_{\mathcal X}$ denotes the orthogonal projection of $\mathcal H$ onto $\mathcal X$.
Moreover, our notation and terminology are standard and, concerning frames in Hilbert spaces, they
are in general those of the book \cite{c} of Christensen.

Recall that a sequence
$\Phi=(\varphi_n)_n\subseteq{{\mathcal H}}$ is a frame for ${\mathcal H}$, if there exist constants $m_\Phi, M_\Phi>0$ such that
\begin{equation}\label{01}
m_\Phi\|f\|^2\leq\sum_{n=1}^\infty|\big<f,\varphi_n\big>|^2\leq M_\Phi\|f\|^2\quad\quad\quad(f\in{\mathcal H}),
\end{equation}
where $m_\Phi, M_\Phi$ are called frame bounds. If only
the right inequality of (\ref{01}) holds, then
$\Phi$ is called a Bessel sequence.
From now on, the notation ${\frak Fr}({\mathcal H})$ is used to denote the set of all frames in $\mathcal H$. Moreover, we define the frame norm on this set as defined below and motivated in Section 4:
$$\|\Phi\|_{{\frak Fr}}:=\|T_{\Phi}\|_{\rm op}.$$

In what follows, for a frame $\Phi$ in $\mathcal H$, the notation $U_\Phi:{\mathcal H}\rightarrow\ell^2$ with $U_\Phi(f):=(\big<f,\varphi_n\big>)_n$ denotes the
associated analysis operator. Its adjoint $T_\Phi:=U^*_\Phi$, the synthesis operator of $\Phi$, maps $\ell^2$
surjectively onto $\mathcal H$ and defined by $T_\Phi c:=\sum_{n=1}^\infty c_n\varphi_n$ for all $c=(c_n)_n\in\ell^2$.
The reader will remark that these operators can be defined for Bessel sequences as for frames.
Observe that $S_\Phi:=T_\Phi U_\Phi$, the frame operator of $\Phi$, is a bounded and positive self-adjoint operator on $\mathcal H$. In particular, each $f\in{\mathcal H}$ can be expressed as
\begin{eqnarray*}\label{02}
f=\sum_{n=1}^\infty\big<f,S_\Phi^{-1}\varphi_n\big>\varphi_n.
\end{eqnarray*}
In particular, if $\Phi$ is not a Riesz basis, then there exists
infinitely many sequences $\Phi^d=(\varphi_n^d)_n$, so-called a dual
of $\Phi$, such that the following reconstruction formula
holds
\begin{equation}\label{03}
f=\sum_{n=1}^\infty\big<f,\varphi_n^d\big>\varphi_n
\quad\quad\quad(f\in{\mathcal H}),
\end{equation}
see Theorem 5.2.3 of \cite{c}.
In terms of the operators
$T_{\Phi}$ and $U_{\Phi^d}$, the equality (\ref{03}) means that
$T_{\Phi}U_{\Phi^d}=Id_{\mathcal H}=T_{\Phi^d}U_{\Phi}$, where here and in
the sequel $Id_{\mathcal H}$ is the identity operator on ${\mathcal H}$.


We conclude this section by recalling the definition and some facts about the gap between two closed subspaces in $\mathcal H$.
Recall from \cite{gita} that for two closed subspace $\mathcal{X}$ and $\mathcal{Y}$ in ${\mathcal H}$ the gap from $\mathcal{X}$ to $\mathcal{Y}$ is defined by
\begin{align*}
\delta (\mathcal{X},\mathcal{Y}):=\big \| (Id_{{\mathcal H}}-P_{\mathcal{Y}})|_{\mathcal{X}}\big \|_{\rm op}=\big\|P_{\mathcal{Y}^{\bot}}|_{\mathcal{X}}\big\|_{\rm op}.
\end{align*}
It is notable that $\delta (\mathcal{X},\mathcal{Y})=\delta (\mathcal{Y}^{\bot},\mathcal{X}^{\bot})$ and
\begin{align*}
\Delta(\mathcal{Y},\mathcal{X}):=\|P_{\mathcal{Y}}-P_{\mathcal{X}}\|_{\rm op}=\max \big \{\delta (\mathcal{X},\mathcal{Y}),\delta (\mathcal{Y},\mathcal{X})\big\}.
\end{align*}
Particularly, if we have $\Delta(\mathcal{Y},\mathcal{X})<1$, then $\mathcal{X}\cap \mathcal{Y}^{\bot}=\{0\}=\mathcal{Y}\cap \mathcal{X}^{\bot}$ and the operators $P_{\mathcal{Y}}\mid_{\mathcal{X}}$ and $P_{\mathcal{X}}\mid_{\mathcal{Y}}$ are isomorphisms.

Finally, we would like to recall the following result from \cite{gita} which will be needed in the sequel.

\begin{lemma}\label{gap-11}
Let $\Phi$ be a frame for $\mathcal H$ and let $\Psi$ be a Bessel sequence. Then
$$\delta({\mathcal R}(U_{\Phi}),\overline{{\mathcal R}(U_{\Psi})})\leq\frac{\|T_\Phi-T_\Psi\|_{\rm op}}{\sqrt{m_\Phi}},$$
where bar denotes the norm closure.
\end{lemma}


\section{Approximately dual frames}

In order to have different reconstruction strategies
various generalizations of Eq. (\ref{03}) have been proposed in the literatures.
One of them considered by Christensen and Laugesen \cite{app} in 2010.
They call two Bessel sequences
$\Phi$ and $\Phi^{ad}$ are
approximately dual frames if $\|Id_{\mathcal H}-T_\Phi U_{\Phi^{ad}}\|_{\rm op}<1$.
The approximation rate of approximately dual frames $\Phi$ and $\Phi^{ad}$ is the number $\varepsilon\in[0,1)$ for which $\|Id_{\mathcal H}-T_\Phi U_{\Phi^{ad}}\|_{\rm op}\leq\varepsilon$.
Later on, Dehghan and Hasankhani-Fard \cite{app2}, stated in \cite{japp} as well as earlier in \cite{wave1}, introduced
and studied the notion of generalized duality for frames in Hilbert spaces.
Recall from \cite[Remark 2.8(i)]{japp} that two frames $\Phi$ and $\Phi^{gd}$ are generalized dual frames, if $T_\Phi U_{\Phi^{gd}}$ is just invertible.

It is shown in \cite{japp}
that approximately
dual frames of $\Phi$ are precisely the sequences
\begin{equation}\label{0720}
\Big({\mathcal A}^*S_\Phi^{-1}\varphi_n+
\Theta^*(\delta_n)\Big)_n,
\end{equation}
where ${\mathcal A}$ is an
operator in $B({\mathcal H})$ for which
$\|Id_{\mathcal H}-{\mathcal A}\|_{\rm op}<1$ and
$$\Theta\in
\verb"ran"_{B({\mathcal H},\ell^2)}(T_\Phi):={\Big\{}\Theta\in B({\mathcal H},\ell^2):~
T_\Phi\Theta=0{\Big\}}.$$
This characterization can be viewed as an operator theoretical variant of a classical
result in \cite{c} for approximately dual frames.

The following notations
will be used frequently in the rest of the paper.

\begin{notation}\label{not}
Let $\Phi$ be a frame for $\mathcal{H}$. For the sake of notational convenience and better citation, in what follows the notation $\Phi_\Theta^{ad}({\mathcal A})$ is used to denote the approximate dual frame of $\Phi$
such that the n'th component, $\pi_n(\Phi_\Theta^{ad}({\mathcal A}))$, is equal to
${\mathcal A}^*S_\Phi^{-1}\varphi_n+
\Theta^*(\delta_n)$.
This says that
\begin{equation}\label{04}
{\mathcal A} f=\sum_{n=1}^\infty\big<f,\pi_n(\Phi_\Theta^{ad}({\mathcal A}))\big>\varphi_n
\quad\quad\quad(f\in{\mathcal H}),
\end{equation}
where we agree to write ${\mathcal A}=T_\Phi U_{\Phi_\Theta^{ad}({\mathcal A})}$.
In the case when ${\mathcal A}=Id_{\mathcal H}$ (alternate dual setting) we use the notation $\widetilde{\Phi}(\Theta)$ for the sequence
$$\Big(\pi_n(\widetilde{\Phi}(\Theta)):=S_\Phi^{-1}\varphi_n+
\Theta^*(\delta_n)\Big)_n.$$
With this notation, $\widetilde{\Phi}(0)$ refers to the canonical dual of $\Phi$ and particularly the synthesis operator of $\Phi_\Theta^{ad}({\mathcal A})$ is
$$T_{\Phi_\Theta^{ad}({\mathcal A})}=
{\mathcal A}^*T_{\widetilde{\Phi}(0)}+\Theta^{*}={\mathcal A}^*S_{\Phi}^{-1}T_{{\Phi}}+\Theta^{*}.$$
Moreover, the letter $\mathcal{AD}(\Phi)$ [resp. $\mathcal{D}(\Phi)$] is used to denote the set of all approximately [resp. alternate] dual frames of $\Phi$.
\end{notation}


If $m^{\rm opt}_\Phi$ refers to the optimal lower frame bound of $\Phi$, that is, the largest $m_\Phi$ that fulfill the corresponding inequality,
then recall from \cite[Proposition 5.4.4]{c} that
$\|U_{\widetilde{\Phi}(0)}\|_{\rm op}^2=(m^{\rm opt}_\Phi)^{-1}$.
The following result generalized this equality to approximately dual frames setting, where the identity operator replaced by an operator $\mathcal A$ with $\|Id_{\mathcal H}-{\mathcal A}\|_{\rm op}<1$. In details, it shows that there is a unique approximately dual
frame of $\Phi$ whose analysis operator obtains the minimal norm of the set
of the norms of analysis operators of all approximately dual frames of $\Phi$.

\begin{proposition}
Let $\Phi$ be a frame for ${\mathcal H}$. Then for any approximately dual frame
$\Phi_{\Theta}^{ad}({\mathcal A})$ of $\Phi$ we have
$$\|U_{\Phi_{\Theta}^{ad}({\mathcal A})}\|_{\rm op}^2\geq\Big(m^{\rm opt}_\Phi\|{\mathcal A}^{-1}\|_{\rm op}^2\Big)^{-1},$$
and $\Phi_{0}^{ad}({\mathcal A})$ is the unique approximately dual frame of $\Phi$ for which $$\|U_{\Phi_{0}^{ad}({\mathcal A})}\|_{\rm op}^2=\Big(m^{\rm opt}_\Phi\|{\mathcal A}^{-1}\|_{\rm op}^2\Big)^{-1}.$$
\end{proposition}
\begin{proof}
By definition, we observe that
$$\|f\|^2\leq\frac{1}{m_\Phi}\|U_{\Phi}f\|_{\ell^2}^2\quad\quad\quad(f\in{\mathcal H}).$$
It follows that
\begin{equation}\label{0157}
\frac{1}{m_\Phi^{\rm opt}}=\inf\Big\{\gamma:~\|f\|^2\leq\gamma\|U_{\Phi}f\|_{\ell^2}^2\quad\forall~f\in{\mathcal H}\Big\}.
\end{equation}
On the other hand, we have
\begin{align*}
\|{\mathcal A}^*f\|^2&=\big<{\mathcal A}{\mathcal A}^*f,f\big>\\
&=\big<T_{\Phi}U_{\Phi_{\Theta}^{ad}({\mathcal A})}T_{\Phi_{\Theta}^{ad}({\mathcal A})}U_{\Phi}f,f\big>\\
&\leq\|U_{\Phi_{\Theta}^{ad}({\mathcal A})}\|_{\rm op}^2\big<T_{\Phi}U_{\Phi}f,f\big>\\
&=\|U_{\Phi_{\Theta}^{ad}({\mathcal A})}\|_{\rm op}^2\|U_{\Phi}(f)\|_{\ell^2}^2.
\end{align*}
From this, by equality
$\|f\|\leq\|{\mathcal A}^{-1}\|_{\rm op}\|{\mathcal A}^*f\|$ ($f\in{\mathcal H}$), we deduce that
$$\|f\|^2\leq\|{\mathcal A}^{-1}\|_{\rm op}^2\|U_{\Phi_{\Theta}^{ad}({\mathcal A})}\|_{\rm op}^2\|U_{\Phi}(f)\|_{\ell^2}^2.$$
We now invoke Eq. (\ref{0157}) to conclude that
$$\|{\mathcal A}^{-1}\|_{\rm op}^2\|U_{\Phi_{\Theta}^{ad}({\mathcal A})}\|_{\rm op}^2\geq\frac{1}{m_\Phi^{\rm opt}}.$$
In order to prove that $\Phi_{0}^{ad}({\mathcal A})$ is the unique approximately dual frame of $\Phi$ for which $$\|U_{\Phi_{0}^{ad}({\mathcal A})}\|_{\rm op}^2=\Big(m^{\rm opt}_\Phi\|{\mathcal A}^{-1}\|_{\rm op}^2\Big)^{-1},$$
we first make use of Douglas' Theorem for surjective operators $T_{\Phi}$ and $\mathcal A$ and find that there exists a unique operator $R:{\mathcal H}\rightarrow\ell^2$ of minimal norm for which ${\mathcal A}=T_\Phi R$, particularly, we have
$$\|R\|_{\rm op}^2=\inf\Big\{\delta:~\|{\mathcal A}^*f\|^2\leq\delta\|Rf\|_{\ell^2}^2\quad\forall~f\in{\mathcal H}\Big\}.$$
On the other hand, an argument similar to the proof of \cite[Lemma 5.4.2]{c} shows that
if $f$ has a representation ${\mathcal A}f=\sum_{n=1}^\infty c_n\varphi_n$ for some coefficients $(c_n)_n$, then
$$\sum_{n=1}^\infty|c_n|^2=\sum_{n=1}^\infty|\big<f,{\mathcal A}^*S_{\Phi}^{-1}\varphi_n\big>|^2+\sum_{n=1}^\infty|c_n-\big<f,{\mathcal A}^*S_{\Phi}^{-1}\varphi_n\big>|^2.$$
It follows that $R=U_{\Phi_{0}^{ad}({\mathcal A})}$ and thus
\begin{align*}
\|U_{\Phi_{0}^{ad}({\mathcal A})}\|_{\rm op}^2&=\inf\Big\{\delta:~\|{\mathcal A}^*f\|^2\leq\delta\|U_{\Phi}f\|_{\ell^2}^2\quad\forall~f\in{\mathcal H}\Big\}\\
&=\inf\Big\{\delta:~\|f\|^2\leq\|{\mathcal A}^{-1}\|_{\rm op}^2\delta\|U_{\Phi}f\|_{\ell^2}^2\quad\forall~f\in{\mathcal H}\Big\}\\
&=\frac{1}{\|{\mathcal A}^{-1}\|_{\rm op}^2}\inf\Big\{\gamma:~\|f\|^2\leq\gamma\|U_{\Phi}f\|_{\ell^2}^2\quad\forall~f\in{\mathcal H}\Big\}\\
&=\frac{1}{m^{\rm opt}_\Phi\|{\mathcal A}^{-1}\|_{\rm op}^2}.
\end{align*}
We have now completed the proof of the theorem.
\end{proof}


The following remark is now immediate.

\begin{remark}
Following alternate dual frames setting, in what follows, the approximately dual frame $\Phi_{0}^{ad}({\mathcal A})$ of $\Phi$ is called the canonical approximately dual frame of $\Phi$.
\end{remark}


It is known that a Riesz basis for $\mathcal H$ is a family of the form $({\mathcal B}e_n)_n$, where $\mathcal B$ is a bijective operator in $B({\mathcal H})$ (see \cite[Definition 3.6.1]{c}). Particularly, if $\Phi$ is a Riesz basis, then its synthesis operator is injective. This together with the characterization (\ref{0720}) imply that the approximately dual frames of a Riesz basis  $\Phi$ such that $\varphi_n={\mathcal B}e_n$ ($n\in{\Bbb N}$), are precisely the sequences
$$\Big({\mathcal A}^*S_\Phi^{-1}{\mathcal B}e_n\Big)_n=\Big({\mathcal A}^*({\mathcal B}^*)^{-1}e_n\Big)_n,$$
where ${\mathcal A}$ is an operator in $B({\mathcal H})$ for which
$\|Id_{\mathcal H}-{\mathcal A}\|_{\rm op}<1$. Hence, an approximately dual frame of a Riesz basis is also a Riesz basis, but, it is not unique.
It is worthwhile to mention that approximately dual frames of a near-Riesz basis are also a near-Riesz basis. Let us recall that a frame $\Phi$ is called a near-Riesz basis whenever it consists of a Riesz basis and a finite number of extra elements. Particularly, the excess of a near-Riesz basis is defined to be the number of elements which have to be removed to obtain a Riesz basis. More generally, by \cite[Corollary 2.7]{excess} approximately dual frames have the same excess, that is,
$$\dim(\ker(T_\Phi))=\dim(\ker(T_{\Phi_{\Theta}^{ad}({\mathcal A})})),$$
for each ${\mathcal A}\in B({\mathcal H})$ with
$\|Id_{\mathcal H}-{\mathcal A}\|_{\rm op}<1$ and
$\Theta\in
\verb"ran"_{B({\mathcal H},\ell^2)}(T_\Phi)$.


\section{The perturbation effect on the duals}

Let us commence by recalling some perturbation conditions of frames in Hilbert spaces and investigate some results related to the Paley and Wiener perturbation Theorem.
\begin{itemize}
\item Following \cite{per2}, we say that the sequence $\Phi$ and $\Psi$ are quadratically close if
$${\rm q}:=\sum_{n=1}^\infty\|\varphi_n-\psi_n\|^2<\infty.$$
\item Inspired by \cite{CHEN}, we say that the frame $\Phi$ and the sequence
$\Psi$ in $\mathcal H$ is d-quadratically close if they are quadratically close with $m_\Phi\leq{\rm q}$ and
\begin{equation}\label{H1134}
{\rm q}_{\Lambda}:=\sum_{n=1}^\infty\|\varphi_n-\psi_n\|
\|\pi_n(\widetilde{\Phi}(\Lambda))\|<\infty,
\end{equation}
for some dual frame $\widetilde{\Phi}(\Lambda)$ of $\Phi$, and they are said to be c-quadratically close whenever \ref{H1134} is satisfied for $\Lambda=0$.
\item As usual we say that a Bessel sequence
$\Psi$ in $\mathcal H$ is a $\mu$-perturbation of
$\Phi$ if $$\|\Phi-\Psi\|_{{\frak Fr}}=\|T_\Phi-T_{\Psi}\|_{\rm op}\leq\mu.$$
\end{itemize}


The following result measures the similarity of a frame and a sequence. More precisely, it shows in particular that the perturbation of a frame remains
being a frame when the perturbation parameter is sufficiently small.

\begin{proposition}\label{per-1200}
Let $\Phi$ be a frame for $\mathcal H$. The following assertions hold for each sequence $\Psi$ in $\mathcal H$.
\begin{enumerate}
\item If $\Phi$ and $\Psi$ are d-quadratically close with ${\rm q}_{\Lambda}<1$, then $\Psi$ is a frame for $\mathcal H$ with bounds $\frac{1}{M_{\widetilde{\Phi}(\Lambda)}}(1-{\rm q}_{\Lambda})^2$ and $M_\Phi\Big(1+\sqrt{\frac{{\rm q}}{M_\Phi}}\Big)^2$. Particularly, if
\begin{enumerate}
\item $\sqrt{m_\Phi M_{\widetilde{\Phi}(\Lambda)}}\leq1-{\rm q}_{\Lambda}$, then   $$\Delta({\mathcal R}(U_\Phi),{\mathcal R}(U_\Psi))\leq\sqrt{\frac{{\rm q}}{m_\Phi}}.$$
\item $\sqrt{m_\Phi M_{\widetilde{\Phi}(\Lambda)}}>1-{\rm q}_{\Lambda}$, then   $$\Delta({\mathcal R}(U_\Phi),{\mathcal R}(U_\Psi))\leq\frac{\sqrt{{\rm q}M_{\widetilde{\Phi}(\Lambda)}}}{1-{\rm q}_{\Lambda}}.$$
\end{enumerate}
\item If $\Phi$ and $\Psi$ are c-quadratically close with ${\rm q}_{0}<1$, then $\Psi$ is a frame for $\mathcal H$ with bounds $m_\Phi(1-{\rm q}_{0})^2$ and $M_\Phi\Big(1+\sqrt{\frac{{\rm q}}{M_\Phi}}\Big)^2$. Particularly, we have
    $$\Delta({\mathcal R}(U_\Phi),{\mathcal R}(U_\Psi))\leq \frac{1}{1-{\rm q}_{0}}\sqrt{\frac{{\rm q}}{m_\Phi}}.$$
    \item If $\Psi$ is a $\mu$-perturbation of
$\Phi$ with $\mu<\sqrt{m_\Phi}$, then $\Psi$ is a frame for $\mathcal H$ with bounds $(\sqrt{m_\Phi}-\mu)^2$ and $(\sqrt{M_\Phi}+\mu)^2$ and particularly
$$\Delta({\mathcal R}(U_\Phi),{\mathcal R}(U_\Psi))\leq\frac{\mu}{\sqrt{m_\Phi}-\mu}.$$
\end{enumerate}
\end{proposition}
\begin{proof}
Assertion (3) is proved in the paper \cite[Theorem 4.6]{gita}. Moreover, assertion (2) is a special case of (1) for the case when $\Theta=0$ and $M_{\widetilde{\Phi}(0)}=\frac{1}{m_\Phi}$. Hence, it will be enough to prove assertion (1). To this end, first note that the frame bounds of $\Psi$ are obtained in the paper \cite[Theorem 2.1]{CHEN}. We now make use of Lemma \ref{gap-11} for $(\Phi,\Psi)$ and $(\Psi,\Phi)$ and find that
$$\delta({\mathcal R}(U_\Phi),{\mathcal R}(U_\Psi))\leq\sqrt{\frac{{\rm q}}{m_\Phi}},$$
and
\begin{align*}
\delta({\mathcal R}(U_\Phi),{\mathcal R}(U_\Psi))&\leq\frac{\sqrt{{\rm q}M_{\widetilde{\Phi}(\Lambda)}}}{1-{\rm q}_\Lambda}\\
&=\sqrt{\frac{{\rm q}}{m_\Phi}}\;\frac{\sqrt{m_\Phi M_{\widetilde{\Phi}(\Lambda)}}}{1-{\rm q}_\Lambda}.
\end{align*}
Hence, the claims follow from the definition of the gap between the closed subspaces ${\mathcal R}(U_\Phi)$ and ${\mathcal R}(U_\Psi)$.
\end{proof}


The following result paves the way for measuring the associated deviations of the canonical and alternate approximately dual frames from the original
ones.

\begin{lemma}\label{dis}
Let $\Phi$ and $\Psi$ be two frames for ${\mathcal H}$ and let
${\mathcal A}_1, {\mathcal A}_2$ be two operators in $B({\mathcal H})$ with $\|Id_{\mathcal H}-{\mathcal A}_i\|_{\rm op}<1$ {\rm(}$i=1, 2${\rm)}. Then we have
\begin{align*}
T_{\Psi_{\Theta_2}^{ad}({\mathcal A}_2)}-T_{\Phi_{\Theta_1}^{ad}({\mathcal A}_1)}&=T_{\Phi_{\Theta_1}^{ad}({\mathcal A}_1)}
(U_\Phi-U_\Psi)T_{\widetilde{\Psi}(0)}+\Theta_2^*\\&-T_{\Phi_{\Theta_1}^{ad}({\mathcal A}_1)}
P_{\ker (T_{\Psi})}-({\mathcal A}_1^*-{\mathcal A}_2^*)T_{\widetilde{\Psi}(0)}.
\end{align*}
\end{lemma}
\begin{proof}
First note that $\ell^2={\mathcal R}(U_{\Psi})\oplus \ker(T_{\Psi})$ and thus $P_{\ker(T_{\Psi})}+P_{{\mathcal R}(U_{\Psi})}= Id_{\ell^2}$,
where $P_X$ denotes the orthogonal projection of $\ell^2$ onto $X$. In particular, $P_{{\mathcal R}(U_{\Psi})}=U_{\Psi}T_{\widetilde{\Psi}(0)}$. Hence, we observe that
\begin{align*}
T_{\Psi_{\Theta_2}^{ad}({\mathcal A}_2)}-T_{\Phi_{\Theta_1}^{ad}({\mathcal A}_1)}&=
{\mathcal A}_2^*T_{\widetilde{\Psi}(0)}+\Theta_2^*-T_{\Phi_{\Theta_1}^{ad}({\mathcal A}_1)}(P_{\ker(T_{\Psi})}+P_{{\mathcal R}(U_{\Psi})})\\
&={\mathcal A}_1^*T_{\widetilde{\Psi}(0)}-T_{\Phi_{\Theta_1}^{ad}({\mathcal A}_1)}P_{{\mathcal R}(U_{\Psi})}+\Theta_2^*\\&-T_{\Phi_{\Theta_1}^{ad}({\mathcal A}_1)}P_{\ker(T_{\Psi})}-({\mathcal A}_1^*-{\mathcal A}_2^*)T_{\widetilde{\Psi}(0)}\\
&=T_{\Phi_{\Theta_1}^{ad}({\mathcal A}_1)}U_{\Phi}T_{\widetilde{\Psi}(0)}-T_{\Phi_{\Theta_1}^{ad}({\mathcal A}_1)}U_{\Psi}T_{\widetilde{\Psi}(0)}+\Theta_2^*\\&-T_{\Phi_{\Theta_1}^{ad}({\mathcal A}_1)}P_{\ker(T_{\Psi})}-({\mathcal A}_1^*-{\mathcal A}_2^*)T_{\widetilde{\Psi}(0)}\\
&=T_{\Phi_{\Theta_1}^{ad}({\mathcal A}_1)}
(U_\Phi-U_\Psi)T_{\widetilde{\Psi}(0)}+\Theta_2^*\\&-T_{\Phi_{\Theta_1}^{ad}({\mathcal A}_1)}
P_{\ker (T_{\Psi})}-({\mathcal A}_1^*-{\mathcal A}_2^*)T_{\widetilde{\Psi}(0)},
\end{align*}
and the lemma is proven.
\end{proof}


The following result estimate the deviation of the canonical approximately dual of original and perturbed sequence. Particularly, it paves the way for computing the distance between the canonical dual and canonical approximately dual of perturbed sequence with respect to the norm $\|\cdot\|_{\frak Fr}$, see Proposition \ref{prop-dis} below.

\begin{proposition}\label{cad}
Let $\Phi$ be a frame for $\mathcal H$ and let $\Psi$ be a $\mu$-perturbation of $\Phi$ with $\mu<\sqrt{m_{\Phi}}$. Then $\Psi$ is a frame for $\mathcal H$ with lower frame bound $(\sqrt{m_\Phi}-\mu)^2$ and particularly, if ${\mathcal A}_1$ and ${\mathcal A}_2$ are two operators in $B({\mathcal H})$ with $\|Id_{\mathcal H}-{\mathcal A}_i\|_{\rm op}<1$ {\rm(}$i=1, 2${\rm)}, then
for the canonical approximate duals $\Phi_{0}^{ad}({\mathcal A}_1)$ and $\Psi_{0}^{ad}({\mathcal A}_2)$ of $\Phi$ and $\Psi$, respectively, we have
\begin{align*}
\Big\|\Phi_{0}^{ad}({\mathcal A}_1)-\Psi_{0}^{ad}({\mathcal A}_2)\Big\|_{{\frak Fr}}&\leq\frac{2\mu\|{\mathcal A}_1\|_{\rm op}}{\sqrt{m_{\Phi}}(\sqrt{m_{\Phi}}-\mu)}+\frac{\|{\mathcal A}_1-{\mathcal A}_2\|_{\rm op}}{\sqrt{m_{\Phi}}-\mu}\\
&\leq\frac{2\mu}{\sqrt{m_{\Phi}}(\sqrt{m_{\Phi}}-\mu)}
+\frac{\varepsilon_1(2\mu+\sqrt{m_\Phi})}{\sqrt{m_{\Phi}}(\sqrt{m_{\Phi}}-\mu)}
+\frac{\varepsilon_2}{\sqrt{m_{\Phi}}-\mu},
\end{align*}
where $\varepsilon_1$ and $\varepsilon_2$ are the approximation rates of $\Big(\Phi,\Phi_{0}^{ad}({\mathcal A}_1)\Big)$ and $\Big(\Psi,\Psi_{0}^{ad}({\mathcal A}_2)\Big)$, respectively.
\end{proposition}
\begin{proof}
If we apply Lemma \ref{dis} for $\Theta_1=0=\Theta_2$, we get
\begin{align*}
T_{\Psi_{0}^{ad}({\mathcal A}_2)}-T_{\Phi_{0}^{ad}({\mathcal A}_1)}&=T_{\Phi_{0}^{ad}({\mathcal A}_1)}
(U_\Phi-U_\Psi)T_{\widetilde{\Psi}(0)}-T_{\Phi_{0}^{ad}({\mathcal A}_1)}
P_{\ker (T_{\Psi})}-({\mathcal A}_1^*-{\mathcal A}_2^*)T_{\widetilde{\Psi}(0)}\\
&={\mathcal A}_1^*T_{\widetilde{\Phi}(0)}
(U_\Phi-U_\Psi)T_{\widetilde{\Psi}(0)}-{\mathcal A}_1^*T_{\widetilde{\Phi}(0)}
P_{\ker (T_{\Psi})}-({\mathcal A}_1^*-{\mathcal A}_2^*)T_{\widetilde{\Psi}(0)}\\
&={\mathcal A}_1^*T_{\widetilde{\Phi}(0)}
(U_\Phi-U_\Psi)T_{\widetilde{\Psi}(0)}-{\mathcal A}_1^*T_{\widetilde{\Phi}(0)}(P_{\ker (T_{\Phi})}-
P_{\ker (T_{\Psi})})\\&-({\mathcal A}_1^*-{\mathcal A}_2^*)T_{\widetilde{\Psi}(0)};
\end{align*}
The reader will remark that in the last equality we use the following fact
$$T_{\widetilde{\Phi}(0)}P_{\ker (T_{\Phi})}=S_{\Phi}^{-1}T_{\Phi}P_{\ker (T_{\Phi})}=0.$$
We now invoke part (3) of Proposition \ref{per-1200} and the equality
$$\Delta(\ker(T_{\Phi}),\ker(T_{\Psi}))=\Delta({\mathcal R}(U_\Phi),{\mathcal R}(U_\Psi))$$
to conclude that
\begin{align*}
\Big\|\Phi_{0}^{ad}({\mathcal A}_1)-\Psi_{0}^{ad}({\mathcal A}_2)\Big\|_{{\frak Fr}}&=
\|T_{\Phi_{0}^{ad}({\mathcal A}_1)}-T_{\Psi_{0}^{ad}({\mathcal A}_2)}\|_{\rm op}\\
&\leq\frac{\mu\|{\mathcal A}_1\|_{\rm op}}{\sqrt{m_{\Phi}}(\sqrt{m_{\Phi}}-\mu)}+\frac{\|{\mathcal A}_1\|_{\rm op}}{\sqrt{m_{\Phi}}}\Delta(\ker(T_{\Phi}),\ker(T_{\Psi}))\\&+\frac{\|{\mathcal A}_1-{\mathcal A}_2\|_{\rm op}}{\sqrt{m_{\Phi}}-\mu}\\
&\leq\frac{2\mu\|{\mathcal A}_1\|_{\rm op}}{\sqrt{m_{\Phi}}(\sqrt{m_{\Phi}}-\mu)}+\frac{\|{\mathcal A}_1-{\mathcal A}_2\|_{\rm op}}{\sqrt{m_{\Phi}}-\mu},
\end{align*}
and the proposition is proven.
\end{proof}


As an immediate consequence we have the following result which study how perturbation effects the canonical dual of original and canonical approximate dual of perturbed
sequence.

\begin{proposition}\label{prop-dis}
Let $\Phi$ be a frame for $\mathcal H$ and let $\Psi$ be a $\mu$-perturbation of $\Phi$ with $\mu<m_{\Phi}$. Then $\Psi$ is a frame for $\mathcal H$ with lower frame bound $(\sqrt{m_\Phi}-\mu)^2$ and particularly, if ${\mathcal A}$ is an operator in $B({\mathcal H})$ with $\|Id_{\mathcal H}-{\mathcal A}\|_{\rm op}<1$, then
for the canonical dual $\widetilde{\Phi}(0)$ of $\Phi$ and canonical approximate dual $\Psi_{0}^{ad}({\mathcal A})$ of $\Psi$, we have
\begin{align*}
\Big\|\widetilde{\Phi}(0)-\Psi_{0}^{ad}({\mathcal A})\Big\|_{\frak Fr}&\leq\frac{2\mu\|{\mathcal A}\|_{\rm op}+\varepsilon\sqrt{m_{\Phi}}}{\sqrt{m_{\Phi}}(\sqrt{m_{\Phi}}-\mu)},
\end{align*}
where $\varepsilon$ is the approximation rate of $\Big(\Psi,\Psi_{0}^{ad}({\mathcal A})\Big)$.
\end{proposition}


Next result shows that if $\Phi$ is a frame and $\Psi$ is a Bessel sequence in $\mathcal H$ which is a $\mu$-perturbation of $\Phi$, then for a given approximately dual
of $\Phi$ one can choose an approximate dual of $\Psi$ such that their frame norm is small
when the perturbation parameter is sufficiently small. Particularly, our choice of the approximately dual of the perturbed frame turns out
to be perfect in terms of best approximations with respect to the norm $\|\cdot\|_{\frak Fr}$.

\begin{theorem}\label{best-app}
Let $\Phi$ be a frame for $\mathcal H$ and let $\Psi$ be a $\mu$-perturbation of $\Phi$, with $\mu<\sqrt{m_{\Phi}}$. Then $\Psi$ is a frame for $\mathcal H$ with lower frame bound $(\sqrt{m_\Phi}-\mu)^2$ and for each $\Theta\in\verb"ran"_{B({\mathcal H},\ell^2)}(T_\Phi)$ and each operators
${\mathcal A}_1, {\mathcal A}_2\in B({\mathcal H})$ with $\|Id_{\mathcal H}-{\mathcal A}_i\|_{\rm op}<1$ {\rm(}$i=1, 2${\rm)},
the approximate dual $\Psi_{\Theta_{ba}}^{ad}({\mathcal A}_2)$ of $\Psi$ is a best approximation
of $\Phi_{\Theta}^{ad}({\mathcal A}_1)$ with respect to the norm $\|\cdot\|_{\frak Fr}$ and a $\lambda$-perturbation of $\Phi_{\Theta}^{ad}({\mathcal A}_1)$, where
$\Theta_{ba}=P_{\ker(T_\Psi)}U_{\Phi_{\Theta}^{ad}({\mathcal A}_1)}$ and
$$\lambda=\frac{\mu}{\sqrt{m_{\Phi}}-\mu}\Big(\frac{\|{\mathcal A}_1\|_{\rm op}}{\sqrt{m_{\Phi}}}+\|\Theta\|_{\rm op}+\frac{\|{\mathcal A}_1-{\mathcal A}_2\|_{\rm op}}{\mu}\Big).$$
\end{theorem}
\begin{proof}
The fact that $\Psi$ is a frame for $\mathcal H$ follows directly from
part (3) of Proposition \ref{per-1200}. This part of theorem also yields that $\Psi$ has the lower frame bound $(\sqrt{m_{\Phi}}-\mu)^2$.
If now we apply lemma \ref{dis} for approximate duals $\Phi_{\Theta}^{ad}({\mathcal A}_1)$
and $\Psi_{\Theta_{ba}}^{ad}({\mathcal A}_2)$ of $\Phi$ and $\Psi$, respectively, we get
$$T_{\Psi_{\Theta_{ba}}^{ad}({\mathcal A}_2)}-T_{\Phi_{\Theta}^{ad}({\mathcal A}_1)}=T_{\Phi_{\Theta}^{ad}({\mathcal A}_1)}
(U_\Phi-U_\Psi)T_{\widetilde{\Psi}(0)}-({\mathcal A}_1^*-{\mathcal A}_2^*)T_{\widetilde{\Psi}(0)}.$$
From this, with inequality $\|T_{\Phi_{\Theta}^{ad}({\mathcal A}_1)}\|_{\rm op}\leq\frac{\|{\mathcal A}_1\|_{\rm op}}{\sqrt{m_{\Phi}}}+\|\Theta\|_{\rm op}$, we can deduce that
the approximate dual $\Psi_{\Theta_{ba}}^{ad}({\mathcal A}_2)$ of $\Psi$ is a
$\lambda$-perturbation of $\Phi_{\Theta}^{ad}({\mathcal A}_1)$.
In order to show that $\Psi_{\Theta_{ba}}^{ad}({\mathcal A}_2)$ is a best approximation
of $\Phi_{\Theta}^{ad}({\mathcal A}_1)$ with respect to the norm $\|\cdot\|_{\frak Fr}$, it suffices to prove that for all $\Lambda\in\verb"ran"_{B({\mathcal H},\ell^2)}(T_\Psi)$ we have
$$\Big\|\Psi_{\Theta_{ba}}^{ad}({\mathcal A}_2)-\Phi_{\Theta}^{ad}({\mathcal A}_1)\Big\|_{{\frak Fr}}\leq\Big\|\Psi_{\Lambda}^{ad}({\mathcal A}_2)-\Phi_{\Theta}^{ad}({\mathcal A}_1)\Big\|_{{\frak Fr}}.$$
To this end, we make use of the equalities $P_{{\mathcal R}(U_\Psi)}+P_{\ker(T_\Psi)}=Id_{\ell^2}$,
$$\Lambda^*P_{{\mathcal R}(U_{\Psi})}=0
\quad\quad\quad{\hbox{and}}\quad\quad\quad
T_{\widetilde{\Psi}(0)}P_{\ker(U_{\Psi})}=S_{\Psi}^{-1}T_{\Psi}P_{\ker(U_{\Psi})}=0$$
to find that
\begin{align*}
T_{\Psi_{\Lambda}^{ad}({\mathcal A}_2)}-T_{\Phi_{\Theta}^{ad}({\mathcal A}_1)}&=
({\mathcal A}_2^*T_{\widetilde{\Psi}(0)}+\Lambda^*)-T_{\Phi_{\Theta}^{ad}({\mathcal A}_1)}\\
&=({\mathcal A}_2^*T_{\widetilde{\Psi}(0)}-T_{\Phi_{\Theta}^{ad}({\mathcal A}_1)})P_{{\mathcal R}(U_{\Psi})}+(\Lambda^*-T_{\Phi_{\Theta}^{ad}({\mathcal A}_1)})P_{\ker(T_{\Psi})}.
\end{align*}
Specially, this equality for $\Theta_{ba}\in \verb"ran"_{B({\mathcal H},\ell^2)}(T_\Psi)$ reduce to the following equality
$$T_{\Psi_{\Theta_{ba}}^{ad}({\mathcal A}_2)}-T_{\Phi_{\Theta}^{ad}({\mathcal A}_1)}=({\mathcal A}_2^*T_{\widetilde{\Psi}(0)}-T_{\Phi_{\Theta}^{ad}({\mathcal A}_1)})P_{{\mathcal R}(U_{\Psi})};$$
This is because of, in this case we have
$$(\Theta_{ba}^*-T_{\Phi_{\Theta}^{ad}({\mathcal A}_1)})P_{\ker(T_{\Psi})}=0.$$
This together with the fact that ${\mathcal R}(U_{\Psi})=\ker(T_{\Psi})^\perp$ implies that
$$\|T_{\Psi_{\Theta_{ba}}^{ad}({\mathcal A}_2)}-T_{\Phi_{\Theta}^{ad}({\mathcal A}_1)}\|_{\rm op}\leq\|T_{\Psi_{\Lambda}^{ad}({\mathcal A}_2)}-T_{\Phi_{\Theta}^{ad}({\mathcal A}_1)}\|_{\rm op},$$
which proves the claim.
\end{proof}


\begin{note} 
For the rest of this paper, we shall use the letter $\Theta_{ba}$
exclusively to denote the operator defined in Theorem \ref{best-app}.
\end{note}


The following discussion together with part 2 of Remark \ref{ex1} below show that the operator $\Theta_{ba}$ is far from devoid of interest and it can have
a nice contribution to frame theory, see also the proof of Theorem \ref{H1157} below.

\begin{remark}
(1) Another perturbation condition of frames in Hilbert spaces is the compactness of the frame synthesis operator which has been investigated by Christensen and Heil \cite[Theorem 4.2]{heil}. In detail, they showed that if $\Phi$ is a frame for $\mathcal H$ and $\Psi$ is a sequence in $\mathcal H$ such that $T_\Phi-T_\Psi$ is compact, then $\Psi$ is a frame sequence. Suppose that either ${\mathcal A}_1={\mathcal A}_2$ or they are compacts.
With an argument similar to the proof of Theorem \ref{best-app} one can show that for each $\Theta\in\verb"ran"_{B({\mathcal H},\ell^2)}(T_\Phi)$
the Bessel sequence $\Psi_{\Theta_{ba}}^{ad}({\mathcal A}_2)$ is the only approximate dual of $\Psi$
such that $T_{\Phi_{\Theta}^{ad}({\mathcal A}_1)}-T_{\Psi_{\Theta_{ba}}^{ad}({\mathcal A}_2)}$ is compact which is also a best approximation
of $\Phi_{\Theta}^{ad}({\mathcal A}_1)$ in the norm space $({\frak Fr}({\mathcal H}),\|\cdot\|_{\frak Fr})$ as well. This is because of, the Banach space of all compact operators on $\mathcal H$ is an ideal of $B({\mathcal H})$.

(2) Theorem \ref{best-app} shows that the canonical approximately dual $\Psi_{0}^{ad}({\mathcal A}_2)$ of $\Psi$ is in general not a best approximation of the canonical dual $\Phi_{0}^{ad}({\mathcal A}_1)$ of $\Phi$.

(3) A special case of Theorem \ref{best-app} gives an explicit construction of the approximately dual frame which is the best approximation of the given alternate dual frame, see also part 2 of Remark \ref{ex1} below.
\end{remark}


It is notable to note that, by replacing $\mu$ with $\sqrt{\rm q}$, the results presented in above can be adapted for two
quadratically close sequences $\Phi$ and $\Psi$ with ${\rm q}<m_\Phi$ as new results; This is because of, in this case $\Psi$ is a $\sqrt{{\rm q}}$-perturbation of
$\Phi$. In the case when we lose the assumption ${\rm q}<m_\Phi$, Christensen
\cite{per2} showed that $\Psi$ is not a frame for the whole space $\mathcal H$ provided that $m_\Phi\leq{\rm q}<\infty$. Hence, in the next two results inspired by \cite{CHEN} we investigate c and d-quadratically close sequences
to formulate similar results as above for quadratically close sequences without imposed condition ${\rm q}<m_\Phi$. Particularly, Example \ref{exam} below shows that these investigations expand our results to the sequences which do not satisfy the Paley and
Wiener perturbation condition.

Since the logic of the proof of the next two results is the same as above, for the conciseness of the presentation we avoid the burden of proof.

\begin{theorem}\label{d-quad}
Let $\Phi$ be a frame for $\mathcal H$ and let $\Psi$ be a sequence in $\mathcal H$. Assume that ${\mathcal A}_1$ and ${\mathcal A}_2$ are operators in $B({\mathcal H})$ with $\|Id_{\mathcal H}-{\mathcal A}_i\|<1$ {\rm(}$i=1, 2${\rm)} and $\Lambda$ is an arbitrary operators in $\verb"ran"_{B({\mathcal H},\ell^2)}(T_\Phi)$. If $\Phi$ and $\Psi$ are d-quadratically close with ${\rm q}_{\Lambda}<1$, then for each $\Theta\in\verb"ran"_{B({\mathcal H},\ell^2)}(T_\Phi)$
the approximate dual $\Psi_{\Theta_{ba}}^{ad}({\mathcal A}_2)$ of $\Psi$ is a best approximation
of $\Phi_{\Theta}^{ad}({\mathcal A}_1)$ in ${\mathcal AD}(\Psi)$ and a $\rho$-perturbation of $\Phi_{\Theta}^{ad}({\mathcal A}_1)$, where
$$\rho=\frac{\sqrt{{\rm q} M_{\widetilde{\Phi}(\Lambda)}}}{1-{\rm q}_\Lambda}\Big(\frac{\|{\mathcal A}_1\|_{\rm op}}{\sqrt{m_{\Phi}}}+\|\Theta\|_{\rm op}+\frac{\|{\mathcal A}_1-{\mathcal A}_2\|_{\rm op}}{\sqrt{{\rm q}}}\Big).$$
Particularly, if
\begin{enumerate}
\item $\sqrt{m_\Phi M_{\widetilde{\Phi}(\Lambda)}}\leq1-{\rm q}_{\Lambda}$, then
\begin{eqnarray*}
\Big\|\Phi_{0}^{ad}({\mathcal A}_1)-\Psi_{0}^{ad}({\mathcal A}_2)\Big\|_{\frak Fr}\leq\frac{2\sqrt{{\rm q}}\|{\mathcal A}_1\|_{\rm op}+\sqrt{m_\Phi}\|{\mathcal A}_1-{\mathcal A}_2\|_{\rm op}}{m_\Phi}.
\end{eqnarray*}
\item  $\sqrt{m_\Phi M_{\widetilde{\Phi}(\Lambda)}}>1-{\rm q}_{\Lambda}$, then
\begin{eqnarray*}
\Big\|\Phi_{0}^{ad}({\mathcal A}_1)-\Psi_{0}^{ad}({\mathcal A}_2)\Big\|_{\frak Fr}\leq\frac{\sqrt{M_{\widetilde{\Phi}(\Lambda)}}}{1-{\rm q}_{\Lambda}}\;\left(2\|{\mathcal A}_1\|_{\rm op}\sqrt{\frac{{{\rm q}}}{{m_\Phi}}}+\|{\mathcal A}_1-{\mathcal A}_2\|_{\rm op}\right).
\end{eqnarray*}
\end{enumerate}
\end{theorem}


Here it should be noted that an interesting version of the following result can be obtain whenever in it the operators $\Lambda$ and $\Theta$ are equal.
Now, let us state Theorem \ref{d-quad} for the case of c-quadratically close sequences.

\begin{theorem}\label{c-quad}
Let $\Phi$ be a frame for $\mathcal H$ and let $\Psi$ be a sequence in $\mathcal H$. Assume that ${\mathcal A}_1$ and ${\mathcal A}_2$ are operators in $B({\mathcal H})$ with $\|Id_{\mathcal H}-{\mathcal A}_i\|<1$ {\rm(}$i=1, 2${\rm)}. If $\Phi$ and $\Psi$ are d-quadratically close with ${\rm q}_{\Lambda}<1$, then for each $\Theta\in\verb"ran"_{B({\mathcal H},\ell^2)}(T_\Phi)$
the approximate dual $\Psi_{\Theta_{ba}}^{ad}({\mathcal A}_2)$ of $\Psi$ is a best approximation
of $\Phi_{\Theta}^{ad}({\mathcal A}_1)$ in ${\mathcal AD}(\Psi)$ and a $\upsilon$-perturbation of $\Phi_{\Theta}^{ad}({\mathcal A}_1)$, where
$\Theta_{ba}=P_{\ker(T_\Psi)}U_{\Phi_{\Theta}^{ad}({\mathcal A}_1)}$ and
$$\upsilon=\frac{\sqrt{{\rm q}}}{\sqrt{m_\Phi}(1-{\rm q}_0)}\Big(\frac{\|{\mathcal A}_1\|_{\rm op}}{\sqrt{m_{\Phi}}}+\|\Theta\|_{\rm op}+\frac{\|{\mathcal A}_1-{\mathcal A}_2\|_{\rm op}}{\sqrt{{\rm q}}}\Big).$$
Particularly, the deviation of the canonical approximately dual of
original and perturbed sequence can be estimated by
\begin{eqnarray*}
\Big\|\Phi_{0}^{ad}({\mathcal A}_1)-\Psi_{0}^{ad}({\mathcal A}_2)\Big\|_{\frak Fr}\leq\frac{1}{\sqrt{m_\Phi}(1-{\rm q}_{0})}\;\left(2\|{\mathcal A}_1\|_{\rm op}\sqrt{\frac{{{\rm q}}}{{m_\Phi}}}+\|{\mathcal A}_1-{\mathcal A}_2\|_{\rm op}\right).
\end{eqnarray*}
\end{theorem}


In the following we will construct an example for which Theorem \ref{c-quad} works while Theorem \ref{best-app} does not. Here it should be noted that it is inspired by \cite[Example 2.5]{CHEN}.

\begin{example}\label{exam}
Let for each $n\in{\Bbb N}$, $\alpha_n=2^n$ and $k_n=4^n$ and let $t_1=3$ and $t_n=2$ for all $n\geq 2$. suppose also that $\Phi$ and $\Psi$ are the following sequences
$$\Phi=\Big(\;\underbrace{\frac{1}{\alpha_n}e_n,\frac{1}{\alpha_n}e_n\cdots,
\frac{1}{\alpha_n}e_n}_{k_n}\;\Big)_n,$$
and
$$\Psi=\Big(\;\underbrace{\frac{t_n}{\alpha_n}e_n,\frac{1}{\alpha_n}e_n\cdots,
\frac{1}{\alpha_n}e_n}_{k_n}\;\Big)_n.$$
It is now not hard to check that $\Phi$ is a tight frame with bounds $m_\Phi=M_\Phi=1$ and thus $S_\Phi=Id_{\mathcal H}$. Moreover, one can easily seen that
$\Psi$ is a frame with bounds $m_\Psi=1$ and $M_\Psi=3$. But, we observe that
$$\|T_\Phi-T_\Psi\|_{\rm op}\geq\|(T_\Phi-T_\Psi)\delta_1\|=\frac{(t_1-1)}{\alpha_1}=1,$$
and
$${\rm q}=\sum_{n=1}^\infty\|\varphi_n-\psi_n\|^2=\sum_{n=1}^\infty\frac{(t_n-1)^2}{\alpha^2}
=1+\sum_{n=2}^\infty\frac{1}{4^n}=\frac{13}{12}>1.$$
Hence, $\Phi$ and $\Psi$ neither quadratically close with ${\rm q}<m_\Phi$ nor $\mu$-perturbation with $\mu<m_\Phi$.
Thus Theorem \ref{best-app} does not work for $\Phi$ and $\Psi$ whereas Theorem \ref{c-quad} works for them. This is because of, we have
$${\rm q}_0=\sum_{n=1}^\infty\|\varphi_n-\psi_n\|\|\pi_n(\widetilde{\Phi}(0))\|
=\sum_{n=1}^\infty\frac{(t_n-1)}{\alpha}\times\frac{1}{\alpha_n}
=\frac{1}{2}+\sum_{n=2}^\infty\frac{1}{4^n}=\frac{7}{12}<1.$$

\end{example}


In the following two result we are going to construct frames from a
given frame and characterize their approximately dual frames.

\begin{theorem}\label{H1157}
Let $\Phi$ be a frame for $\mathcal H$ and let $\Psi$ be a $\mu$-perturbation of $\Phi$ such that $\mu<\sqrt{m_{\Phi}}/2$. Then then there exists a one-to-one correspondence between ${\mathcal AD}(\Phi)$ and ${\mathcal AD}(\Psi)$.
\end{theorem}
\begin{proof}
In order to achieve the proof of theorem we show that the map
$$\Gamma:{\mathcal AD}(\Phi)\rightarrow{\mathcal AD}(\Psi);\quad
\Phi^{ad}_{\Theta}({\mathcal A})\mapsto\Psi^{ad}_{\Theta_{ba}}({\mathcal A})$$
is bijective. To this end, first suppose that $\Gamma(\Phi^{ad}_{\Theta}({\mathcal A}))=\Gamma(\Phi^{ad}_{\Theta'}({\mathcal A}'))$. Hence, we have
\begin{equation}\label{1257}
\Psi^{ad}_{\Theta_{ba}}({\mathcal A})=\Psi^{ad}_{\Theta'_{ba}}({\mathcal A}').
\end{equation}
It follows that $U_{\Psi^{ad}_{\Theta_{ba}}({\mathcal A})}=U_{\Psi^{ad}_{\Theta'_{ba}}({\mathcal A}')}$ and thus we have
$$U_{\widetilde{\Psi}(0)}{\mathcal A}+P_{\ker(T_{\Psi})}U_{\Phi^{ad}_{\Theta}({\mathcal A})}=U_{\widetilde{\Psi}(0)}{\mathcal A}'+P_{\ker(T_{\Psi})}U_{\Phi^{ad}_{\Theta'}({\mathcal A}')}.$$
Now, using the duality relation between $\Psi$ and $\widetilde{\Psi}$ we get ${\mathcal A}={\mathcal A}'$. From this, by Eq. (\ref{1257}), we deduce that
$P_{\ker(T_{\Psi})}U_{\Phi^{ad}_{\Theta}({\mathcal A})}=P_{\ker(T_{\Psi})}U_{\Phi^{ad}_{\Theta'}({\mathcal A})}$. Hence, we get $P_{\ker(T_{\Psi})}\Theta=P_{\ker(T_{\Psi})}\Theta'$. We now invoke the equalities
$$P_{\ker(T_{\Phi})}\Theta=\Theta\quad\quad\quad{\hbox{and}}\quad\quad\quad P_{\ker(T_{\Phi})}\Theta'=\Theta'$$
to conclude that
\begin{align*}
(P_{\ker(T_{\Psi})}|_{\ker(T_{\Phi})})\Theta\delta_n
&=P_{\ker(T_{\Psi})}\Theta\delta_n\\
&=P_{\ker(T_{\Psi})}\Theta'\delta_n\\
&=(P_{\ker(T_{\Psi})}|_{\ker(T_{\Phi})})\Theta'\delta_n,
\end{align*}
for all $n\in{\Bbb N}$. On the other hand, by part 3 of Proposition \ref{per-1200}, we observe that $\Delta(\ker(T_\Phi),\ker(T_\Psi))<1$ and thus the operator $P_{\ker(T_{\Psi})}|_{\ker(T_{\Phi})}$ is isomorphism, by what was mention in Section 2. It follows that $\Theta=\Theta'$ and therefore the map $\Gamma$ is injective. It remains to show that $\Gamma$ is surjective. For this, suppose that $\Psi^{ad}_{\Lambda}({\mathcal A})$ is an arbitrary approximate dual of $\Psi$. If we set
$$\Theta:=P_{\ker(T_{\Phi})}(P_{\ker(T_{\Psi})}|_{\ker(T_{\Phi})})^{-1}
(\Lambda-P_{\ker(T_{\Psi})}U_{\widetilde{\Phi}(0)}{\mathcal A}),$$
then we observe that $\Gamma(\Theta)=\Lambda$ and thus
$$\Gamma(\Phi^{ad}_{\Theta}({\mathcal A}))=\Psi^{ad}_{\Gamma(\Theta)}({\mathcal A})$$
which implies that $\Gamma$ is surjective.
\end{proof}


\section{Application to Gabor frames}

Recall that a Gabor frame is a frame for $L^2({\Bbb R})$ of the
form ${\mathcal G}:=(E_{mb}T_{na}g)_{m,n\in{\Bbb Z}}$, where $a, b>0$,
$g\in L^2({\Bbb R})$, $T_{na}f(x)=f(x-na)$ and $E_{mb}f(x)=e^{2\pi
imbx}f(x)$ for all $f\in L^2({\Bbb R})$. In view of \cite[Theorem
11.3.1]{c}, the sequence ${\mathcal G}$  can only be a frame if
$ab\leq 1$, but it is not a sufficient condition. Another necessary condition can be expressed in terms of the boundedness of the function
$G(x):=\sum_{n\in{\Bbb Z}}|g(x-na)|^2$. More precisely, if $\mathcal G$ is a Bessel sequence with bound $M_{\mathcal G}$, then
\begin{equation}\label{Gelar0714}
G(x)\leq bM_{\mathcal G}\quad\quad a.e.~x\in{\Bbb R},
\end{equation}
see Proposition 11.3.4 of \cite{c}. Now, let us to recall a sufficient condition from \cite[Theorem 11.4.2]{c}. If $g\in L^2({\Bbb R})$, $a, b>0$ and suppose that
$$m_{\mathcal G}:=\frac{1}{b}\sup_{x\in[0,a]}\sum_{k\in{\Bbb Z}}\left|\sum_{n\in{\Bbb Z}}g(x-na)\overline{g(x-na-k/b)}\right|<\infty,$$
and
$$M_{\mathcal G}:=\frac{1}{b}\inf_{x\in [0,a]}\left[\sum_{n\in{\Bbb Z}}|g(x-na)|^2-\sum_{k\neq 0}\Big|\sum_{n\in{\Bbb Z}}g(x-na)\overline{g(x-na-k/b)}\Big|\right]>0,$$
then $\mathcal G$ is a frame for $L^2({\Bbb R})$ with bounds $m_{\mathcal G}, M_{\mathcal G}$.

Recall also from \cite{c} that if
$\mathcal G$ is a frame and $ab<1$, then there exists infinitely many
$g^d$ in $L^2({\Bbb R})$ such that we have the following
reconstruction formula for each $f\in L^2({\Bbb R})$
$$f=\sum_{m,n\in{\Bbb Z}}\big<f,E_{mb}T_{na}g^d\big>E_{mb}T_{na}g;$$
that is, the Gabor frames ${\mathcal G}$ and ${\mathcal G}^d=(E_{mb}T_{na}g^d)_{m,n\in{\Bbb Z}}$ are dual Gabor frames.
But the
standard choice of $g^d$ is $S_{\mathcal G}^{-1}g$, where
$$S_{\mathcal G}:L^2({\Bbb R})\rightarrow L^2({\Bbb R});\quad f\mapsto\sum_{m,n\in{\Bbb Z}}\big<f,E_{mb}T_{na}g\big>E_{mb}T_{na}g$$ is the frame operator of
$\mathcal G$. It is worth mentioning
that if the operator ${\mathcal
A}\in B(L^2({\Bbb R}))$ commutes with $E_{\pm b}$ and $T_{\pm a}$,
then ${\mathcal A}$ and its adjoint commute with $E_{mb}$ and $T_{na}$
for all $m, n\in{\Bbb Z}$. In particular, Lemma 12.3.1 of \cite{c}
guarantees that there are infinitely many operators $\mathcal A$ on
$L^2({\Bbb R})$ which commute with $E_{\pm b}$, $T_{\pm a}$ and
$\|Id_{L^2({\Bbb R})}-{\mathcal A}\|<1$. For example, it is sufficient to
set $\mathcal A$ equal to an appropriate scalar multiple of the
frame operator of $\mathcal G$ or ${\mathcal A}:=T_{{\mathcal G}}U_{{\mathcal G}^{ad}}$. Moreover, we would like to recall from \cite[Proposition 12.3.6]{c} that if $g\in L^2({\Bbb R})$, $a, b>0$ are given and $\mathcal G$ is a frame for $L^2({\Bbb R})$, then a Gabor system ${\mathcal G}^d=(E_{mb}T_{na}g^d)_{m,n\in{\Bbb Z}}$ is a dual frame if and only if the function $g^d$ has the form
$$
g^d=S_{\mathcal G}^{-1}g+h-\sum_{m,n\in{\Bbb Z}}\big<S_{\mathcal G}^{-1}g,E_{mb}T_{na}g\big>E_{mb}T_{na}h
$$
for some function $f\in L^2({\Bbb R})$ for which $(E_{mb}T_{na}h)_{m,n\in{\Bbb Z}}$ is a Bessel sequence.


The following remark will be needed in the sequel.

\begin{remark}\label{ex1}
Let ${\mathcal G}=(E_{mb}T_{na}g)_{m,n\in{\Bbb Z}}$ and ${\mathcal G}^d=(E_{mb}T_{na}g^d)_{m,n\in{\Bbb
Z}}$ be dual Gabor frames.
\begin{enumerate}
\item An argument similar to the proof of \cite[Proposition 12.3.6]{c} with the aid of
\cite[Theorem 2.1]{japp} shows that
the Gabor system ${\mathcal G}^{ad}=(E_{mb}T_{na}g^{ad})_{m,n\in{\Bbb Z}}$ is an approximately dual frame if and only if the function $g^{ad}$ has the form
\begin{equation}\label{adgabor}
g^{ad}={\mathcal A}^*S_{\mathcal G}^{-1}g+h-\sum_{m,n\in{\Bbb Z}}\big<S_{\mathcal G}^{-1}g,E_{mb}T_{na}g\big>E_{mb}T_{na}h
\end{equation}
for some function $f\in L^2({\Bbb R})$ for which $(E_{mb}T_{na}h)_{m,n\in{\Bbb Z}}$ is a Bessel sequence and an operator $\mathcal A\in B(L^2({\Bbb R}))$ which commutes with $E_{\pm b}$ and $T_{\pm a}$
and $\|Id_{L^2({\Bbb R})}-{\mathcal A}\|<1$.
If in Eq. (\ref{adgabor}) we set $h=S_{\mathcal G}g^d$, then we obtain the following generator of an approximately dual Gabor frame of $\mathcal G$
$$g^{ad}={\mathcal A}^* S_{\mathcal G}^{-1}g-g+S_{\mathcal G}(g^d),$$
which is very applicable for constructing of approximately dual Gabor frames with a desired approximation rate, see \cite[Section 3]{japp}.
\item An argument similar to the proofs of \cite[Lemma 6.3.6 and Theorem 6.3.7]{c} and \cite[Theorem 2.1]{japp} implies that if $\Phi^{ad}$ is an approximately dual of $\Phi$, then
there exists an operator ${\mathcal A}\in B({\mathcal H})$ with $\|Id_{\mathcal H}-{\mathcal A}\|_{\rm op}<1$ and
$W\in B(\ell^2,{\mathcal H})$ such that for operator $V:=W(Id_{\ell^2}-U_\Phi S^{-1}_\Phi T_{\Phi})$ we have
\begin{align}\label{DBessel}
\varphi_n^{ad}&={\mathcal A}^*S_\Phi^{-1}{\varphi}_n+V\delta_n\nonumber\\
&={\mathcal A}^*S_\Phi^{-1}{\varphi}_n+W\delta_n-\sum_{j=1}^\infty\big<S_\Phi^{-1}{\varphi}_n,
\varphi_j\big>W\delta_j.
\end{align}
Particularly, if we set $\Omega=(\omega_n:=W\delta_n)_n$, then $T_\Omega=W$ and thus the boundedness of $W$ implies that $\Omega$ is a Bessel sequence. Hence, this characterization of approximately duals of a given frame $\Phi$ says that there exists correspondence between Bessel sequences in $\mathcal H$ and the duals of $\Phi$. But, it is not hard to check that $Id_{\ell^2}-U_\Phi S^{-1}_\Phi T_{\Phi}=P_{\ker(T_\Phi)}$.
It follows that $V=WP_{\ker(T_\Phi)}$ and therefore we have
\begin{align}\label{DOperator}
\varphi_n^{ad}={\mathcal A}^*S_\Phi^{-1}{\varphi}_n+(P_{\ker(T_\Phi)}W^*)^*\delta_n=\pi_n({\Phi}_{P_{\ker(T_\Phi)}W^*}^{ad}({\mathcal A})).
\end{align}
From Eqs. (\ref{DBessel}) and (\ref{DOperator}), we deduce that in Theorem \ref{best-app} above $\Phi_\Theta^{ad}({\mathcal A}_1)$ is the Bessel sequence corresponding to the approximately dual frame $\Psi_{\Theta_{ba}}^{ad}({\mathcal A}_2)$ of $\Psi$. That is, in terms of Bessel sequence we have the following explicit representation for the component of the sequence $\Psi_{\Theta_{ba}}^{ad}({\mathcal A}_2)$
$$\pi_n(\Psi_{\Theta_{ba}}^{ad}({\mathcal A}_2))={\mathcal A}_2^*S_\Psi^{-1}{\psi}_n+\pi_n(\Phi_\Theta^{ad}({\mathcal A}_1))-\sum_{j=1}^\infty\big<S_\Psi^{-1}{\psi}_n,
\psi_j\big>\pi_n(\Phi_\Theta^{ad}({\mathcal A}_1)).$$
\end{enumerate}
\end{remark}


Now we are in position to consider the effect of perturbations of Gabor systems on their alternate and approximately dual Gabor frames.

Our starting point is the investigation of the perturbation question on the generating function of a Gabor system. Its proof can be obtained via Proposition \ref{cad}, Theorem \ref{best-app} and
Eq. (\ref{Gelar0714}) combined with Theorem 22.4.1 of \cite{c} and Remark \ref{ex1} above. The details are omitted.

\begin{theorem}\label{Gabor1}
Let $g_1, g_2\in L^2({\Bbb R})$ and $a, b>0$ be given, and suppose that $\mathcal G$ is a frame for $L^2({\Bbb R})$. If
$${\rm r}:=\frac{1}{b}\sup_{x\in[0,a]}\sum_{k\in{\Bbb Z}}\left|\sum_{n\in{\Bbb Z}}(g_1-g_2)(x-na)\overline{(g_1-g_2)(x-na-k/b)}\right|<m_{{\mathcal G}_1},$$
then ${\mathcal G}_2$ is a frame for $L^2({\Bbb R})$ with lower frame bound $(\sqrt{m_{{\mathcal G}_1}}-\sqrt{{\rm r}})^2$ and it is a $\sqrt{{\rm r}}$-perturbation of ${\mathcal G}_1$ as well.
Moreover, if ${\mathcal A}_1$ and ${\mathcal A}_2$ are two operators in $B(L^2({\Bbb R}))$ which commutes with $E_{\pm b}$ and $T_{\pm a}$ and $\|Id_{\mathcal H}-{\mathcal A}_i\|_{\rm op}<1$ {\rm(}$i=1, 2${\rm)}, then we have
\begin{align*}
\sum_{n\in{\Bbb Z}}\left|({\mathcal A}_1^*S_{{\mathcal G}_1}^{-1}g_1-{\mathcal A}_2^*S_{{\mathcal G}_2}^{-1}g_2)(x-na)\right|^2\leq b\left[\frac{2\sqrt{{\rm r}}\|{\mathcal A}_1\|_{\rm op}}{\sqrt{m_{{\mathcal G}_1}}(\sqrt{m_{{\mathcal G}_1}}-\sqrt{{\rm r}})}+\frac{\|{\mathcal A}_1-{\mathcal A}_2\|_{\rm op}}{\sqrt{m_{{\mathcal G}_1}}-\sqrt{{\rm r}}}\right],
\end{align*}
for almost everywhere $x\in{\Bbb R}$, and for given approximately dual Gabor frames ${\mathcal G}_1^{ad}$ of $\mathcal G$ with generating function
$$g_1^{ad}={\mathcal A}_1^*S_{{\mathcal G}_1}^{-1}g_1+h-\sum_{m,n\in{\Bbb Z}}\big<S_{{\mathcal G}_1}^{-1}g_1,E_{mb}T_{na}g_1\big>E_{mb}T_{na}h,$$
the function
$$g_2^{ad}={\mathcal A}_2^*S_{{\mathcal G}_2}^{-1}g_2+g_1^{ad}-\sum_{m,n\in{\Bbb Z}}\big<S_{{\mathcal G}_2}^{-1}g_2,E_{mb}T_{na}g_2\big>E_{mb}T_{na}g_1^{ad}$$
generates the best approximation of ${\mathcal G}_1^{ad}$ in ${\mathcal AD}({\mathcal G}_2)$.
\end{theorem}


The following three corollaries are now
immediate. The first ones state Theorem \ref{Gabor1} especially for the case of canonical approximately dual Gabor frames and the second ones study how perturbation in the Wiener space norm effects the alternate and approximately dual Gabor frames of original and perturbed generating function.

\begin{corollary}\label{Gabor2}
Let $g_1, g_2, a, b, {\rm r}, {\mathcal A}_1$ and ${\mathcal A}_2$ be as in Theorem \ref{Gabor1}. Then the function
$$g_2^{ad}={\mathcal A}_2^*S_{{\mathcal G}_2}^{-1}g_2+{\mathcal A}_1^*S_{{\mathcal G}_1}^{-1}\left(g_1^{ad}-\sum_{m,n\in{\Bbb Z}}\big<S_{{\mathcal G}_2}^{-1}g_2,E_{mb}T_{na}g_2\big>E_{mb}T_{na}g_1^{ad}\right)$$
generates the best approximation of the canonical approximately dual Gabor frame of ${\mathcal G}_1$ in ${\mathcal AD}({\mathcal G}_2)$ and
$$\sum_{n\in{\Bbb Z}}\left|({\mathcal A}_1^*S_{{\mathcal G}_1}^{-1}g_1^{ad}-g_2^{ad})(x-na)\right|\leq \frac{b\sqrt{{\rm r}}}{\sqrt{m_{{\mathcal G}_1}}-\sqrt{{\rm r}}}\left(\frac{\|{\mathcal A}_1\|_{\rm op}}{\sqrt{m_{{\mathcal G}_1}}}+\frac{\|{\mathcal A}_1-{\mathcal A}_2\|_{\rm op}}{\sqrt{{\rm r}}}\right)$$
for almost everywhere $x\in{\Bbb R}$.
\end{corollary}


Let us recall from \cite{c} that for given $a>0$ the Wiener space is defined by
$${W}:=\left\{g:{\Bbb R}\rightarrow{\Bbb C}|~g~ {\hbox{is~ measurable~ and~}} \|g\|_{W,a}:=\sum_{k\in{\Bbb Z}}\|g\chi_{[ka,(k+1)a)}\|_\infty<\infty\right\},$$
where $\chi_{[ka,(k+1)a)}$ denotes the characteristic function of $[ka,(k+1)a)$ on $\Bbb R$. This space equipped with the norm $\|\cdot\|_{W,a}$ becomes a Banach space.
It is notable to mention that the space $W$ is independent of the choice of $a$ and different choices give equivalent norms. It is shown in \cite[Corollary 22.4.2]{c} that if $g_1, g_2\in L^2({\Bbb R})$ and $a, b>0$ are such that ${\mathcal G}_1$ is a frame for $L^2({\Bbb R})$ and $\|g_1-g_2\|_{W,a}<\sqrt{\frac{bm_{{\mathcal G}_1}}{2}}$, then ${\rm r}\leq\frac{2}{b}\|g_1-g_2\|_{W,a}$. Hence, the following next result follows from Theorem \ref{Gabor1} with ${\rm r}$ replaced by $\frac{2}{b}\|g_1-g_2\|_{W,a}$.

\begin{corollary}\label{Gabor3}
Let $g_1, g_2\in L^2({\Bbb R})$ and $a, b>0$ be given, and suppose that $\mathcal G$ is a frame for $L^2({\Bbb R})$. If
$\|g_1-g_2\|_{W,a}<\sqrt{\frac{bm_{{\mathcal G}_1}}{2}}$, then ${\mathcal G}_2$ is a frame for $L^2({\Bbb R})$ with lower frame bound $\left(\sqrt{m_{{\mathcal G}_1}}-\sqrt{\frac{2}{b}}\;\|g_1-g_2\|_{W,a}\right)^2$ and it is a $\sqrt{\frac{2}{b}}\|g_1-g_2\|_{W,a}$-perturbation of ${\mathcal G}_1$ as well.
Moreover, if ${\mathcal A}_1$ and ${\mathcal A}_2$ are two operators in $B(L^2({\Bbb R}))$ which commutes with $E_{\pm b}$ and $T_{\pm a}$ and $\|Id_{\mathcal H}-{\mathcal A}_i\|_{\rm op}<1$ {\rm(}$i=1, 2${\rm)}, then we have
\begin{align*}
\sum_{n\in{\Bbb Z}}\left|({\mathcal A}_1^*S_{{\mathcal G}_1}^{-1}g_1-{\mathcal A}_2^*S_{{\mathcal G}_2}^{-1}g_2)(x-na)\right|^2&\leq \frac{2\sqrt{2b}\|g_1-g_2\|_{W,a}\|{\mathcal A}_1\|_{\rm op}}{\sqrt{m_{{\mathcal G}_1}}(\sqrt{m_{{\mathcal G}_1}}-\sqrt{\frac{2}{b}}\|g_1-g_2\|_{W,a})}\\&+\frac{b\|{\mathcal A}_1-{\mathcal A}_2\|_{\rm op}}{\sqrt{m_{{\mathcal G}_1}}-\sqrt{\frac{2}{b}}\|g_1-g_2\|_{W,a}},
\end{align*}
for almost everywhere $x\in{\Bbb R}$, and for given approximately dual Gabor frames ${\mathcal G}_1^{ad}$ of $\mathcal G$ with generating function
$$g_1^{ad}={\mathcal A}_1^*S_{{\mathcal G}_1}^{-1}g_1+h-\sum_{m,n\in{\Bbb Z}}\big<S_{{\mathcal G}_1}^{-1}g_1,E_{mb}T_{na}g_1\big>E_{mb}T_{na}h,$$
the function
$$g_2^{ad}={\mathcal A}_2^*S_{{\mathcal G}_2}^{-1}g_2+g_1^{ad}-\sum_{m,n\in{\Bbb Z}}\big<S_{{\mathcal G}_2}^{-1}g_2,E_{mb}T_{na}g_2\big>E_{mb}T_{na}g_1^{ad}$$
generates the best approximation of ${\mathcal G}_1^{ad}$ in ${\mathcal AD}({\mathcal G}_2)$.
\end{corollary}


In the case of canonical approximately dual Gabor frames, Corollary \ref{Gabor3} reduces to the next result.

\begin{corollary}\label{Gabor4}
Let $g_1, g_2, a, b, {\mathcal A}_1$ and ${\mathcal A}_2$ be as in Corollary \ref{Gabor3}. Then the function
$$g_2^{ad}={\mathcal A}_2^*S_{{\mathcal G}_2}^{-1}g_2+{\mathcal A}_1^*S_{{\mathcal G}_1}^{-1}\left(g_1^{ad}-\sum_{m,n\in{\Bbb Z}}\big<S_{{\mathcal G}_2}^{-1}g_2,E_{mb}T_{na}g_2\big>E_{mb}T_{na}g_1^{ad}\right)$$
generates the best approximation of the canonical approximately dual Gabor frame of ${\mathcal G}_1$ in ${\mathcal AD}({\mathcal G}_2)$ and
$$\sum_{n\in{\Bbb Z}}\left|({\mathcal A}_1^*S_{{\mathcal G}_1}^{-1}g_1^{ad}-g_2^{ad})(x-na)\right|\leq \frac{b\sqrt{{2}}\|g_1-g_2\|_{W,a}}{\sqrt{bm_{{\mathcal G}_1}}-\sqrt{2}\|g_1-g_2\|_{W,a}}\left(\frac{\|{\mathcal A}_1\|_{\rm op}}{\sqrt{m_{{\mathcal G}_1}}}+\frac{\sqrt{b}\|{\mathcal A}_1-{\mathcal A}_2\|_{\rm op}}{\sqrt{2}\|g_1-g_2\|_{W,a}}\right)$$
for almost everywhere $x\in{\Bbb R}$.
\end{corollary}




\end{document}